\documentclass{amsart}
\usepackage{amsmath, amsthm, amsfonts, amssymb}
\usepackage{mathtools}

\usepackage{hyperref}

\hypersetup{
	colorlinks=true,
	linkcolor=blue,
	citecolor=red,
}

\newtheorem{defn}{Definition}[section]
\newtheorem{theorem}[defn]{Theorem}
\newtheorem{lemma}[defn]{Lemma}
\newtheorem{prop}[defn]{Proposition}

\newtheorem{example}[defn]{Example}

\newcommand{\calA}{\mathcal{A}}
\newcommand{\calB}{\mathcal{B}}
\newcommand{\calD}{\mathcal{D}}
\newcommand{\calE}{\mathcal{E}}
\newcommand{\calJ}{\mathcal{J}}
\newcommand{\calK}{\mathcal{K}}
\newcommand{\calL}{\mathcal{L}}
\newcommand{\calS}{\mathcal{S}}
\newcommand{\calU}{\mathcal{U}}
\newcommand{\calV}{\mathcal{V}}
\newcommand{\calW}{\mathcal{W}}

\DeclareMathOperator{\rk}{\mathrm{rk}}
\DeclareMathOperator{\cork}{\mathrm{co-rk}}

\begin{document}
\title{Stability of certain Engel-like Distributions}

\author{Aritra Bhowmick}

\address{Statistics and Mathematics Unit, Indian Statistical Institute\\ 203,
B.T. Road, Calcutta 700108, India.\\ e-mail:
avowmix@gmail.com}

\begin{abstract}
In this article we introduce a higher dimensional analogue of Engel structure, motivated by the Cartan prolongation of contact manifolds. We study the stability of such structure, generalizing the Gray-type stability for Engel manifolds.
\end{abstract}

\maketitle
\date{}

\section{Introduction}
In \cite{montGeneric} Montgomery proved that a generic rank $r$ distribution on a manifold of dimension $n$ is not stable if $r(n-r) > n$. Among the cases that are excluded by this inequality are line fields (when $r=1$), contact and even contact structures (when $r=n-1$) and lastly Engel structures (when $r=2, n=4$). An Engel structure is a rank $2$ distribution $\calD$ on a $4$-manifold $M$ such that $\calD^2$ is rank $3$ distribution and $\calD^3=TM$. Like contact structures, any Engel structure is locally given as the common kernel of two $1$-forms (see \cite{montEngelDeform}) $$dz-ydx, \qquad dy - wdx$$
But unlike contact structures, Engel structures are not stable under arbitrary isotopy. In fact, any Engel structure $\calD$ defines a complete flag $$\calL\subset\calD\subset\calE,$$
where the line field $\calL$, called the characteristic line field, is usually not stable under isotopy. Golubev proved a modified version of Gray-type theorem for Engel structures in \cite{golubev}.\\

Engel manifolds are closely related to contact 3-manifolds. Starting with a 3-dimensional manifold with a contact structure $\xi$, one obtains a circle bundle by Cartan prolongation of $\xi$, where the total space of the bundle carries an Engel structure $\mathcal D$ with its characteristic line field tangent to fibers (\cite{montEngelDeform}). Prolongation on an arbitrary contact manifold $(N^{2n+1},\xi)$ give rise to fiber bundles $M\to N$ with fiber $\mathbb R P^{n-1}$. The total space of the bundle supports a flag $\calL\subset\calD\subset\calE\subset TM$ on $M$, with rank vector $(2n-1, 2n, 4n-1, 4n)$, where
\[\calD^2=\calE,\ \calD^3=TM,\]
and $\calL$ is the Cauchy characteristic distribution (see~\ref{defnCauchyDist}) of $\calE$. In general, if we have a flag $\mathcal D\subset \calE$ satisfying $\calD^2=\calE$ and $\calD^3=TM$, then it does not necessarily follow that the Cauchy characteristic distribution $\mathcal L$ is contained in $\mathcal D$ (Example~\ref{exNonExample}).
Motivated by these examples, we introduce the notion of generalized Engel structures on a manifold $M$. These are
distributions $\calD$ of even co-rank, such that $\calE=\calD^2$ is a co-rank $1$ distribution, $\calD^3 = TM$
and the Cauchy characteristic distribution $\calL$ of $\calE$ is contained in $\calD$ and has co-rank $1$ in $\calD$.
The distribution $\mathcal L$ is referred as the characteristic distribution of $\mathcal D$. The main goal of this article is to demonstrate a Gray-type stability property of these distributions
\begin{theorem}
	\label{thmGrayGeneral}
Let $\calD_t$, $0\leq t\leq 1$, be a smooth one-parameter family of generalized Engel distributions on a closed manifold $M$. Assume that the characteristic distribution $\calL_t$ of $\calD_t$ is independent of $t$ and say $\calL=\calL_t$ for all $t$. Then there exists an isotopy $\phi_t$ of $M$ such that $$\phi_{t*}\calD_t = \calD_0,\qquad \phi_{t*}\calL = \calL$$
\end{theorem}
We also obtain a local normal form for a set of generators of the annihilating ideal of a generalized Engel distribution $\mathcal D$.\\

The article is organized as follows: In Section 2 we recall some basic notions about distributions. In Section 3, we introduce generalized Engel structures and describe the Pfaffian system defining them. In Section 4 and 5 we prove the main results of this article.

\section{Basic Notions and Examples}
Given any distribution $\calA$ on a manifold $M$, we can think of it as a sheaf of local sections of the sub-bundle $\calA\subset TM$. Given two distributions $\calA,\calB$, define by $[\calA,\calB]$ as the sheaf of vector fields obtained by taking Lie brackets of local sections. Using this notation, recursively define, $$\calD^{i+1}=\calD + [\calD, \calD^i], \qquad \calD^1 = \calD$$
At every $x\in M$ we have the integer $q_i(x)=\dim \calD^i_x$, where $\calD^i_x$ is the stalk at the point $x$. Note that $\calD^i$ defines a distribution if the integer $q_i(x)$ is locally constant. The integer sequence $(q_i(x))_i$ is called the \textit{growth vector} for the distribution $\calD$ at $x$. A distribution is \emph{regular} if the growth vector is independent of the point $x$. A regular distribution $\calD$ is called \textit{nonholonomic} if there is an integer $k$ such that $TM=\calD^k$. In this article, we only consider nonholonomic distributions in the above sense.

Before moving onto some examples, we recall the definition of Cauchy characteristic distribution (\cite{bryantExteriorDiff}), as it will play an important role in understanding the generalized Engel distributions.
\begin{defn}
	\label{defnCauchyDist}
	{\em Given a co-rank $1$ distribution $\calE$ on a manifold $M$, consider the collection, $$\calL = \Big\{X\in\calE \Big| [X,Y] \in\calE \text{ for all }Y\in \calE\Big\}$$ If $\calL$ has constant rank everywhere it is called the \textit{Cauchy characteristic distribution} of $\calE$.}
\end{defn}
We can locally define $\calL$ as follows. Suppose, $\calE\underset{loc.}{=}\ker\theta$. Then, $$\calL\underset{loc.}{=}\ker d\theta|_\calE=\Big\{X\in\calE \Big| d\theta(X,Y)=0, \forall Y\in\calE\Big\}$$
It is easy to see that Cauchy characteristic distribution is integrable. Indeed if $X,Y\in \calL$ and $Z\in\calE$, then we have, $$[[X,Y],Z] = [X,[Y,Z]] - [Y, [X,Z]]$$
Now, $[X,Z],[Y,Z]\in\calE$ and hence $[[X,Y],Z]\in\calE$. Thus $[X,Y]\in\calL$. But then $\calL$ is integrable by Frobenius Theorem.

\begin{example}\quad

\em{
\begin{enumerate}
	\item[(a)] A contact distribution $\xi$ on an odd-dimensional manifold $M$ is a co-rank $1$ distribution such that $\xi^2=TM$ and the Cauchy characteristic distribution of $\xi$ is trivial.
	
	\item[(b)] Similarly, an even contact structure on an even-dimensional manifold $M$ is a co-rank $1$ distribution $\calE$ such that $\calE^2=TM$ and the Cauchy characteristic distribution of $\calE$ is a line field. Like contact structures, an even contact structure $\calE$ is locally given as the the kernel of some $1$-form $\alpha$ satisfying $\alpha\wedge d\alpha^n \ne 0$. \\
	
	\item[(c)] An Engel structure $\calD$ is a co-rank $2$ nonholonomic distribution on a $4$-dimensional manifold $M$, such that $\calD^2$ is an even contact structure and $\calD^3=TM$. The characteristic line field $\calL$ of $\calD^2$ turns out to be contained in $\calD$ (see \cite{montEngelDeform}). Thus the Engel structure completely defines the flag $\calL\subset\calD\subset\calD^2\subset TM$. Any Engel structure $\calD$ can be locally realized as the kernel of two $1$-forms, $$dz-ydx,\qquad dy-wdx.$$
\end{enumerate}
}
\end{example}

We are particularly interested in distributions in higher dimensions, which exhibit properties similar to Engel structures.

\subsection{Cartan Prolongation}
A prime example of Engel manifold appears as Cartan prolongation of contact $3$-manifolds $(M,\xi)$. On $M$ we construct a fiber bundle, with total space $\mathbb{P}(\xi)$ and fibers $\mathbb{RP}^1$. There is a canonical Engel distribution on the total space, where the characteristic line field is along the fibers. We describe below the Cartan prolongation of an arbitrary contact manifold.\\

Consider an odd dimensional manifold $N^{2n+1}$ with a contact structure $\xi$. On $N$ we construct the Grassmann bundle $$\mathbb{RP}^{2n-1} \hookrightarrow \mathbb{P}\xi \overset{\pi}{\to} N,$$
wher the fiber over a point $x\in N$ is the projective space of lines in the vector space $\xi_x$. The total space $Q=\mathbb{P}\xi$ is of dimension $4n$. The inverse image of $\xi$ defines a corank $1$ distribution $\calE = d\pi^{-1}(\xi)$ on $Q$ which is easily seen to be an even contact structure. On the other hand, there is a distribution $\mathcal D$ which is obtained as follows: At a point $[\ell]\in Q$, where $\ell$ is a line in $\xi_p$ for $p\in N$, define $\calD_{[\ell]} = d\pi|_{[l]}^{-1}(\ell)$. Since $\pi$ is a submersion, $\calD$ is a co-rank $2n$ distribution on $Q$. Clearly, $\calD\subset\calE$. Set $\calL$ as the vertical sub-bundle of $TQ$ over $N$, i.e., $\calL$ is tangent along the fibers. Thus, we have a flag, $\calL\subset\calD\subset\calE$. The distribution $\calD$ is called the prolongation of $\xi$. In particular, if $n=1$ and $N$ is a contact $3$ manifold, then $\dim Q=4$ and $\calD$ is an Engel structure on $Q$.\\

We can now observe a few general properties of this flag. Since $\xi$ is locally expressed as $\ker (dz - \sum_{i=1}^n y_idx_i)$ we have
\[\xi \underset{loc.}{=} \Big\langle\partial_{y_i}, P_i = \partial_{x_i} + y_i\partial_z \Big| i=1\ldots n\Big\rangle\]
Any line $\ell\subset\xi_p$ is represented by a no-trivial linear combination of these vectors. Hence, on $Q$ we can introduce new homogeneous coordinates along the fiber as
\[\{a_1,\ldots,a_n,b_1,\ldots,b_{n-1}\}.\]
At any point on $Q$, we can associate a uniquely defined vector,
\[Z = P_n + \sum_{i=1}^n a_i \partial_{y_i} + \sum_{j=1}^{n-1} b_jP_j.\]
Then we can describe the flag $\calL\subset\calD\subset\calE$ locally as follows,
\begin{align*}
\calL &= \langle \partial_{a_1},\ldots,\partial_{a_n},\partial_{b_1},\ldots,\partial_{b_{n-1}} \rangle\\
\calD &= \calL \oplus \langle Z \rangle\\
\calE &= \calD \oplus \langle \partial_{y_1},\ldots,\partial_{y_n},P_1,\ldots,P_{n-1} \rangle\\
TQ &= \calE \oplus \langle \partial_z \rangle
\end{align*}
From this description we observe that
\begin{itemize}
	\item $\cork\calD$ is even, $\cork\calE$ is $1$ and co-rank of $\calL$ in $\calD$ is $1$
	\item $\calD^2=\calE, \calD^3=TM$
	\item $\calL$ is the Cauchy characteristic distribution $\calE$ so that $[\calL,\calE]\subset\calE$.
\end{itemize}

\section{Generalized Engel structure}
Motivated by the Cartan prolongation of a contact structure, we define a generalized Engel structure.
\begin{defn}
	\label{defnGenEngel}
	{\em A \emph{generalized Engel structure} or an Engel-like distribution on a manifold $M$ is a distribution $\calD$ of even co-rank, such that
	\begin{enumerate}
		\item $\calE=\calD^2$ is a co-rank $1$ distribution
		\item $\calD^3 = TM$
		\item $\calL$, the Cauchy characteristic distribution of $\calE$, is contained in $\calD$
		\item $\calL$ has co-rank $1$ in $\calD$.
	\end{enumerate}
Thus, we have a flag
	\[
	\rlap{$\underbrace{\phantom{\calL\subset\calD}}_{\cork=1}$}\calL\subset
	\overbrace{\calD\subset\rlap{$\underbrace{\phantom{\calE\subset TM}}_{\cork=1}$}\calE \subset TM}^{\textrm{even }\cork}
	\]
	The distribution $\calL$ will be called the \emph{characteristic distribution} of the generalized Engel distribution $\calD$}.
\end{defn}

\subsection{A remark on the definition}
When $\dim M=4$ and $\calD$ is of co-rank $2$, we have an Engel structure. As mentioned earlier, in this case the Cauchy characteristic distribution $\calL$ of $\calE=\calD^2$ is completely determined by $\mathcal D$ and it is contained in $\calD$. For higher co-rank we can not expect this to happen in general and this can be seen from the examples below. In the first two examples $\calL\not\subset\calD$ and in the third one co-rank of $\calL$ in $\calD$ is not of co-rank $1$. All the examples are constructed over $\mathbb{R}^8$, where the coordinates are understood from the context.

\begin{example}\label{exNonExample}\quad

\em{
\begin{enumerate}
	\item[(a)] Suppose, $\mathcal{D} = \langle\partial_x,\partial_y,\partial_z, \partial_w + x\partial_{x_1} + y\partial_{y_1} + z\partial_{z_1} + z_1\partial_t\rangle$.
	Then, $[\mathcal{D},\mathcal{D}] = \langle\partial_{x_1},\partial_{y_1},\partial_{z_1}\rangle$
	and hence, $$\mathcal{E} = \mathcal{D}^2 = \langle\partial_x,\partial_y,\partial_z,\partial_{x_1},\partial_{y_1},\partial_{z_1},\partial_w + z_1\partial_t\rangle$$
	Lastly, $[\mathcal{D},\mathcal{D}^2] = \langle \partial_t\rangle$ and so $\mathcal{D}^3=TM$. The Cauchy characteristic distribution of $\calE$ is $\mathcal{L} = \langle\partial_x,\partial_y,\partial_z,\partial_{x_1},\partial_{y_1}\rangle$ which is a rank $5$ distribution. Clearly in this case we have $\calL\not\subset\calD$.
	\item[(b)] Consider, $\mathcal{D} = \langle\partial_w, \partial_{x_1} + w\partial_{y_1} + y_1\partial_z, \partial_{x_2} + w\partial_{y_2} + y_2\partial_z, \partial_{x_3} + w\partial_{y_3}\rangle$
	Then, $[\mathcal{D},\mathcal{D}] = \langle\partial_{y_1},\partial_{y_2},\partial_{y_3}\rangle$ and so, $$\mathcal{E} = \mathcal{D}^2 = \langle\partial_w, \partial_{y_1},\partial_{y_2},\partial_{y_3}, \partial_{x_1} + y_1\partial_z, \partial_{x_2} + y_2\partial_z, \partial_{x_3}\rangle$$
	Clearly, $\mathcal{D}^3=TM$. Also, $\mathcal{L} = \langle\partial_w,\partial_{y_3},\partial_{x_3}\rangle$. Since $\partial_{y_3}\not\in\mathcal{D}$, we have $\calL\not\subset\calD$.
	\item[(c)] Let $v_i = \partial_{x_i} + w\partial_{y_i} + y_i\partial_z$ for $i=1,2,3$ and $\mathcal{D} = \langle\partial_w,v_1,v_2,v_3\rangle$. Then, $[\mathcal{D},\mathcal{D}] = \langle\partial_{y_1},\partial_{y_2},\partial_{y_3}\rangle$ and so $\mathcal{E}=\mathcal{D}^2$ is a co-rank 1 distribution and $\mathcal{D}^3 = TM$. Also, $\mathcal{L} = \langle\partial_w\rangle$. In this case, we have the flag $\calL\subset\calD\subset\calE$, where $\calE$ is an even contact structure. Further, note that there exists a co-rank $1$ integrable distribution $\langle v_1,v_2,v_3\rangle$ contained in $\calD$.
\end{enumerate}
}
\end{example}

The above examples justify the conditions (3) and (4) in the definition of the generalized Engel structure.

\subsection{Pfaffian system}
A \emph{Pfaffian system} is a sub-distribution of the cotangent bundle $T^*M$. Given a distribution $\calD\subset TM$, we have an associated Pfaffian system $\calS(\calD)$ defined as the collection of $1$-forms which vanish on $\calD$. In this section we would like to find out the Pfaffian system for a generalized Engel distribution.\\

We start with a co-rank $k+1$ generalized Engel distribution $\calD$, where $k=2l+1$ is odd. Suppose locally,
\[\calE = \{\theta=0\} \ \ \text{ and }\ \ \ \calD=\{\omega^1=\ldots=\omega^k=0=\theta\}\]
for $1$-forms $\theta,\omega^1,\ldots,\omega^k$. Set, $\eta^i = \omega^1\wedge\ldots\omega^k\wedge\theta\wedge d\omega^i$.
\begin{prop}
\label{propDistToForms}
We have the following.
\begin{enumerate}
	\item \label{itemForms:1} $\{\eta^1,\ldots,\eta^k\}$ is point-wise linearly independent
	\item \label{itemForms:2} $\omega^i\wedge\theta\wedge d\theta^{l+1} = 0, \forall i=1,\ldots,k$
	\item \label{itemForms:3} $\theta\wedge d\theta^{l+1}\ne 0$
	\item \label{itemForms:4} $\theta\wedge d\theta^{l+2} = 0$
\end{enumerate}
\end{prop}
\begin{proof}
Choose local vector fields $D$ and $R$ such that $\calD/\calL = \langle D \mod \calL \rangle$ and $TM/\calE =\langle R \mod \calE\rangle$. Since $\calL$ is integrable and $\calE=\calD^2 =\calD + [\calD,\calD]$, we have that the map
\begin{align*}
\calL &\to \calE/\calD\\
L &\mapsto [D,L] \mod \calD
\end{align*}
is a surjective bundle map. Since $\{\omega^i\}$ is linearly independent and $\calD$ is their common kernel in $\calE$, we can choose dual vectors $V^i\in\calE/\calD$. Also from the surjectivity, there exists $L^i\in \calL$ such that $V^i=[D,L^i] \mod \calD$. Then $\eta^i\ne 0\forall i$, since we have, $$\eta^i(V^1,\ldots,V^k,R,D,L^i) \ne 0$$
If possible, let $\{\eta^i\}$ be linearly dependent at the point $p$. Then without loss of generality we may assume that, $\eta^1 = \sum_{i=2}^k f_i \eta^i$ at $p$ for some functions $f_i$. Set, $\tilde{\omega}^1 = \omega^1 - \sum_{i>1} f_i\omega^i$. Then clearly, $\calD$ is also defined as $\{\tilde{\omega}^1=\omega^2=\ldots=\omega^k=0=\theta\}$. But then we must have that
$$\tilde{\omega}^1\wedge\omega^2\wedge\ldots\omega^k\wedge\theta\wedge d\tilde{\omega}^1 \ne 0$$
On the other hand,
\begin{align*}
\tilde{\omega}^1\wedge\omega^2\wedge\ldots\omega^k\wedge\theta\wedge d\tilde{\omega}^1 &= \Big(\omega^1-\sum_{i>1}f_i\omega^i\Big)\wedge\omega^2\wedge\ldots\wedge\omega^k\wedge\theta\wedge \Big(d\omega^1 - \sum_{i>1}d(f_i\omega^i)\Big)\\
&=\omega^1\wedge\ldots\omega^k\wedge\theta\wedge\Big(d\omega^1-\sum_{i>1}df_i\wedge\omega^i -\sum_{i>1}f_id\omega^i\Big)\\
&=\omega^1\wedge\ldots\omega^k\wedge\theta\wedge\Big(d\omega^1-\sum_{i>1}f_id\omega^i\Big)\\
&=\eta^1-\sum_{i>1}f_i\eta^i\\
&=0, \textnormal{ at the point } p
\end{align*}
This is a contradiction. Hence we have that the set of $(k+3)$-forms $\{\eta^i\}$ are point-wise linearly independent. This proves (\ref{itemForms:1}).

Now observe that $\calL=\ker (d\theta|_{\ker\theta})$. So, on $\calE/\calL$ we have that $d\theta$ has full rank. Since $\calL$ is of co-rank $k+2=2l+3$, $$\theta\wedge d\theta^{l+1}\ne 0, \theta\wedge d\theta^{l+2} = 0$$
Thus proving (\ref{itemForms:3}) and (\ref{itemForms:4}).

Next, consider the $2l+3$-form $\omega^i\wedge d\theta^{l+1}|_\calE$ on $\calE$. For any $L\in\calL$ we have that, $\iota_L\omega^i\wedge d\theta^{l+1}|_\calE$ identically zero, since $\omega^i(L) = 0$ and $\iota_Ld\theta|_\calE=0$. Thus, $\calL$ is in the kernel of $\omega^i\wedge d\theta^{l+1}|_\calE$. But $\calL$ has co-rank $2l+2$ in $\calE$ and then by simple rank counting argument, $\omega^i\wedge d\theta^{l+1}|_\calE$ is identically zero. But $\calE = \ker\theta$ and hence $\omega^i\wedge\theta\wedge d\theta^{l+1}= 0$, proving (\ref{itemForms:2}).

\end{proof}

The converse of \ref{propDistToForms} is also true. Suppose we are given some co-rank $k+1$ distribution $\calD$ on a manifold $M$, where $k=2l+1$, such that $\calD$ is locally the common kernel of $1$-forms $\{\theta,\omega^1,\ldots,\omega^k\}$, satisfying
\begin{itemize}
	\item $\{\eta^1,\ldots,\eta^k\}$ is point-wise linearly independent, where $\eta^i = \omega^1 \wedge \ldots \omega^k \wedge \theta \wedge d\omega^i$
	\item $\omega^i\wedge\theta\wedge d\theta^{l+1} = 0$ for all $i=1,\ldots,k$
	\item $\theta\wedge d\theta^{l+1}\ne 0$
	\item $\theta\wedge d\theta^{l+2} = 0$
\end{itemize}
\begin{prop}
Under the above hypotheses, $\calD$ is a generalized Engel structure.
\end{prop}
\begin{proof}
Set $\calE = \ker\theta$ and $\calL = \ker d\theta|_\calE$. These are locally defined distributions of co-rank $1$ and $2l+3$ respectively. We can get a local framing of $TM/\calL$ as $\big\{R, X_i,Y_j \big| i,j = 1,\ldots,l+1\big\}$ such that
$\theta\wedge d\theta^{l+1}(R,X_1,Y_1,\ldots,X_{l+1},Y_{l+1}) \ne 0$. Consider $L\in \calL$. Then we have, $$0=\omega^i\wedge\theta\wedge d\theta^{l+1}(L, R, X_1,\ldots,Y_{l+1})= \omega^i(L)\theta\wedge d\theta^{l+1}(R, X_1,\ldots,Y_{l+1})$$
since all other terms vanish. But then $\omega^i(L)=0$, for all $i$. Thus, $L\in\ker D$. So we have the flag, $$\calL\subset\calD\subset\calE$$

Next we show $\calE=\calD^2$. First note that $\calD \underset{loc.}{=} \calL \oplus \langle Z \rangle$, for some choice of vector field $Z$. Since $\calL$ is Cauchy characteristic distribution of $\calE$, we have $[\calL,\calE] \subset\calE$. In particular, $[\calL, Z] \subset \calE$. Also $\calL$ being integrable, we have $[\calL,\calL] \subset \calL$ by Frobenius. Then clearly, $[\calD,\calD] \subset \calE$. Thus $\calD^2\subset\calE$. For the equality, consider the map,
\begin{align*}
\Phi : \calL &\to \calE/\calD\\
L &\mapsto [Z,L] \mod \calD
\end{align*}
$\Phi$ is a bundle map : $[Z,fL] = Z(f)L + f[Z,L] \equiv f[Z,L] \mod \calD$, as $\calL\subset\calD$. We will show that $\Phi$ is of full rank, which will imply that $\calE=\calD^2$. Equivalently this happens if $\Phi^*$ is injective. Dualizing $\Phi$ we get,
\begin{align*}
\Phi^* : \big(\calE/\calD\big)^* &\rightarrow \calL^*\\
[\alpha] &\mapsto -\iota_Zd\alpha|_\calL
\end{align*}
where $(\calE/\calD)^*$ consists of classes of $1$-forms $\alpha$ defined on $\calE$, which annihilates $\calD$. Consider the $1$-forms, $$\tau^i \coloneqq -\iota_Z d\omega^i|_\calL$$ defined on $\calL$. Since $\omega^i$ induces a basis for $(\calE/\calD)^*$, it is enough to show that the maps $\tau^i$ are point-wise linearly independent for $\Phi^*$ to be injective. If not, then without loss of generality assume, $\tau^1 = \sum_{i>1} f_i\tau^i$ at some point $p$, for some functions $f_i$. Get dual vectors $\{R,V_1,\ldots,V_k\}$ in $TM/\calD$ of $\{\theta,\omega^1,\ldots,\omega^k\}$ respectively. Now, for any $L\in\calL$, we have $\eta^i(V_1,\ldots,V_k,R,Z,L)=d\omega^i(Z,L)=-\tau^i(L)$. Thus, $$\eta^1(V_1,\ldots,V_k,R,Z,L)=-\sum_{i>1}f_i\eta^i(V_1,\ldots,V_k,R,Z,L)$$ at the point $p$. But then, $\eta^1=-\sum_{i>1}f_i\eta^i$ at $p$, contradicting point-wise linear independence of $\{\eta^i\}$. Hence, $\{\tau^i\}$ must be linearly independent point-wise. Thus we get $\calE=\calD^2$. $\calE$ and consequently $\calL$ are now globally defined distributions.

Lastly to verify $\calD^3=TM$ note that, $d\theta$ is non-degenerate on $\calE/\calL$. In particular, for $Z\in\calD$ satisfying $\calD = \calL\oplus\langle Z \rangle$, we have $\iota_Zd\theta \ne 0$. So, $d\theta(Z,V) \ne 0$, for some $V\in\calE/\calD$. Then $V \in [\calD,\calD]$ and $0\ne d\theta(Z,V)=-\theta[Z,V]$. Thus, $TM = \calE \oplus \langle [Z,V] \rangle$. So, $TM = \calD^3$.
\end{proof}

\section{Stability of generalized Engel structure}
Engel structures are not globally stable due to the presence of an integrable subbundle, though they have local stability property. Golubev proved the following Gray-type theorem for Engel structure which shows that a homotopy $\mathcal D_t$, $0\leq t\leq 1$, of Engel structures is obtained by an isotopy provided the characteristics distribution of $\mathcal D_t$ is independent of $t$.
\begin{theorem}[\cite{golubev}]\label{thmGolubev}
Let $\calD_t$, $0\leq t\leq 1$, be a one-parameter family of oriented Engel structures on an oriented compact $4$-dimensional manifold $M$, such that the characteristic line field $\calL(\calD_t)=\calL$ for all $t$. Then there exists an isotopy $\phi_t$, $0\leq t\leq 1$, of $M$ such that 
\[\phi_{t*}(\calD_t)=\calD_0,\qquad \phi_{t*}(\calL)=\calL.\]
\end{theorem}
Theorem ~\ref{thmGrayGeneral} is direct generalization of Theorem~\ref{thmGolubev} for generalized Engel structure. We shall first prove a special case of of this theorem.
\begin{theorem}
\label{thmIsotopyOfD}
Suppose $\mathcal D_t$, $0\le t \le 1$, is a one-parameter family of generalized Engel structure on a closed manifold $M$ such that $\calD_t^2$ is independent of $t$ and equals $\calE$. If $\mathcal L$ is the Cauchy characteristics distribution of $\mathcal E$ then there exists an isotopy $\phi_t$  of $M$ such that 
\[\phi_{t*}\calD_0 = \calD_t, \qquad \phi_{t*}\calE = \calE,  \qquad \phi_{t*}\calL = \calL.\]
\end{theorem}

\subsection{Proof of Theorem~\ref{thmIsotopyOfD}}
The approach of the proof is very similar to that of Adachi in \cite{adachiCorank}. The proof will follow through a sequence of Lemmas. We assume that $\calD_t$ is a given smooth one-parameter family of co-rank $k+1$, where $k=2l+1$, generalized Engel distribution on a closed manifold $M$ such that $\calE=\calD_t^2$ is independent of $t$ and the Cauchy characteristic distribution of $\calE$ is $\calL$.\\

Suppose $\mathcal D_t$ is defined as the common kernel of the 1-forms $\omega^1_t,\omega^2_t,\ldots,\omega^k_t,\theta$ on some open subset $U$ of $M$. Let $X_t$ be a time-dependent vector field on $M$ which satisfies the following system of equations on $U$.

\begin{equation}\iota_{X_t}d\omega^i_t|_{\calD_t} + \frac{d}{dt}\omega^i_t|_{\calD_t} = 0, \forall i=1,\ldots,k; \qquad X_t\in\calL.\label{eqnODE}\end{equation}
As we shall see below, if any such $X_t$ exists, then its flow has the desired property. We begin with the following observation.
\begin{prop}
\label{propDiffEqLocalIndept} Suppose $\calE=\ker{\eta}$ and $\calD_t=\{\mu^1_t=\ldots=\mu^j_t=0=\eta\}$ for a smooth family of $1$-forms $\{\mu^i_t,\eta\}$ on $U$. Then the vector fields $X_t$, $0\leq t\leq 1$, in (\ref{eqnODE}) also satisfies the relations
\[\iota_{X_t}d\mu^i_t|_{\calD_t} + \frac{d}{dt}\mu^i_t|_{\calD_t} = 0.\]
In other words, $X_t$ depends on the distributions $\mathcal D_t$, not on the choice of local 1-forms defining the distributions.
\end{prop}
\begin{proof}
Suppose $\{\mu^i_t,\eta\}$ and $\{\omega^i_t, \theta\}$ be as above. Then we must have that $\eta=f\theta$ for some non-zero function $f$ and
\[
\begin{pmatrix}
\mu^1_t\\
\vdots\\
\mu^k_t
\end{pmatrix}= A_t
\begin{pmatrix}
\omega^1_t\\
\vdots\\
\omega^k_t
\end{pmatrix}
\]
for a family of non-singular $k\times k$ matrix $A_t=(A^{ij}_t)$. So, $\mu^i_t=\sum_j A^{ij}_t\omega^j_t$. Then $d\mu^i_t=\sum_j dA^{ij}_t \wedge \omega^j_t + A^{ij}_t d\omega^j_t$. So,
\begin{align*}
\iota_{X_t}d\mu^i_t|_{\calD_t} &= \sum_j \iota_{X_t}dA^{ij}_t\omega^j_t|_{\calD_t} - \iota_{X_t}\omega^j_t dA^{ij}_t|_{\calD_t} + A^{ij}_t\iota_{X_t}d\omega^j_t|_{\calD_t}\\
&= \sum_j A^{ij}_t\iota_{X_t}d\omega^j_t|_{\calD_t}, \textnormal{as } \omega^j_t(X_t)=0 \textnormal{ and } \omega^j_t|_{\calD_t}=0\\
&= -\sum_j A^{ij}_t \frac{d}{dt}\omega^j_t|_{\calD_t}
\end{align*}
On the other hand,
\begin{align*}
\frac{d}{dt}\mu^i_t|_{\calD_t} &= \sum_j \frac{dA^{ij}_t}{dt}\omega^j_t|_{\calD_t} + A^{ij}_t\frac{d}{dt}\omega^j_t|_{\calD_t}\\
&=\sum_j A^{ij}_t\frac{d}{dt}\omega^j_t|_{\calD_t}
\end{align*}
Hence we have, $\iota_{X_t}d\mu^i_t|_{\calD_t} + \frac{d}{dt}\mu^i_t|_{\calD_t} = 0$. Thus $X_t$ is a solution for every family of local forms defining $\calD_t$.
\end{proof}

\begin{prop}
\label{propDiffEquationLocal} Suppose there exists a time dependent vector field $X_t$, $0\leq t\leq 1$, on $M$ which satisfies the following conditions:
\[\iota_{X_t}d\omega^i_t|_{\calD_t} + \frac{d}{dt}\omega^i_t|_{\calD_t} = 0, \forall i=1,\ldots,k; \qquad X_t\in\calL\]
where $\theta$, $\omega^i_t, 0\leq t\leq 1, i=1,2,\dots,k$ is a smooth family of (local) 1-forms such that \[\calE\underset{loc.}{=}\ker\theta \ \ \  and  \ \ \ \calD_t\underset{loc.}{=}\{\omega^1_t = \ldots = \omega^k_t = 0 = \theta\}.\]
Then the flow $\phi_t$ obtained by integrating the time-dependent vector field $X_t$ satisfies, \[\phi_{t*}\calD_0=\calD_t,\ \ 
\phi_{t*}\calL=\calL\]
\end{prop}
\begin{proof}
Since $X_t\in\calL$ and $[\calL,\calE]\subset\calE$, we have that $\phi_{t*}\calE=\calE$ and hence $\phi_{t*}\calL=\calL$, as $\calL$ is completely defined by $\calE$. So we have, $\phi_t^*\theta=F_t\theta$ for some family of non-vanishing functions $F_t$.

In order to verify that $\phi_{t*}\calD_0=\calD_t$, we would show the existence of smooth families of functions $G^{ij}_t$ and $F^i_t$, satisfying \[\phi_{t}^*\omega^i_t = \sum_{j} G^{ij}_t\omega^j_0 + F^i_t\theta,\forall i \tag{*} \label{eqnMoser}\]
where the matrix $$\Big(G^{ij}_t\Big)_{k\times k}$$ is non-singular, for every $t$. We solve for $$\frac{d}{dt}\Big(\phi_{t}^*\omega^i_t - \sum_{j} G^{ij}_t\omega^j_0 - F^i_t\theta\Big)=0, \forall i$$
any solution of which will satisfy (\ref{eqnMoser}). Differentiating both sides of (\ref{eqnMoser}) with respect to $t$ we get,
\begin{align*}
\sum_j \frac{dG^{ij}_t}{dt}\omega^j_0 + \frac{dF^i_t}{dt}\theta &= \frac{d}{dt}\phi_t^*\omega^i_t\\
&=\phi_t^*\Big(L_{X_t}\omega^i_t + \frac{d}{dt}\omega^i_t\Big)\\
&=\phi_t^*\Big(\iota_{X_t}d\omega^i_t + \frac{d}{dt}\omega^i_t\Big)
\end{align*}
Now from the hypothesis, $\iota_{X_t}d\omega^i_t|_{\calD_t} + \frac{d}{dt}\omega^i_t|_{\calD_t} = 0$ and hence we must have family of functions $g^{ij}_t$ and $f^i_t$ such that,
$$\iota_{X_t}d\omega^i_t + \frac{d}{dt}\omega^i_t = \sum_j g^{ij}_t \omega^j_t + f^i_t\theta$$
Pulling back by $\phi_t$ we have,
\begin{align*}
\phi_t^*\Big(\iota_{X_t}d\omega^i_t + \frac{d}{dt}\omega^i_t\Big) &= \sum_j (g^{ij}_t\circ\phi_t) \phi_t^*\omega^j_t + (f^i_t\circ\phi_t) \phi_t^*\theta\\
&=\sum_j(g^{ij}_t\circ\phi_t)\Big(\sum_pG^{jp}_t\omega^p_0 + F^j_t\theta\Big) + (f^i_t\circ\phi_t) F_t\theta\\
&=\sum_p \Big(\sum_j (g^{ij}_t\circ\phi_t)G^{jp}_t\Big)\omega^p_0 + \Big(\sum_j (g^ij_t\circ\phi_t)F^j_t + (f^i_t\circ\phi_t)F_t\Big)\theta
\end{align*}
Comparing coefficients of $\omega^i_0$ and $\theta$, we get a system of first order differential equations,
\begin{align*}
\frac{dG^{ip}_t}{dt} &= \sum_j (g^{ij}_t\circ\phi_t)G^{jp}_t, \forall i,p\\
\frac{dF^i_t}{dt} &= \sum_j (g^{ij}_t\circ\phi_t)F^j_t + (f^i_t\circ\phi_t)F_t, \forall i
\end{align*}
with the initial conditions, $$G^{ij}_0 = \partial_{ij},\qquad F^i_0 = 0$$
This system, being affine, has a solution for all $0\le t\le 1$. Since the matrix $(G^{ij}_t)=I_{k}$ at $t=0$, we must have that the matrix is non-singular in a range around $0$. Hence we have that $\phi_{t*}\calD_0=\calD_t,\phi_{t*}\calE=\calE$
\end{proof}

Next we shall discuss how to obtain the time dependent vector field $X_t$ satisfying the hypothesis of Proposition~\ref{propDiffEquationLocal}. First, suppose that we have a (locally) finite open cover $\{U^\lambda\}$ of $M$ and local fields $X^\lambda_t\in\calL|_{U_\lambda}$ on $U_\lambda$, which satisfy the relations \[\iota_{X^\lambda_t}d\omega^{i,\lambda}_t|_{\calD_t} + \frac{d}{dt}\omega^{i,\lambda}_t|_{\calD_t} = 0, \forall i=1,\ldots,k\]
on $U_\lambda$,  where $\calD_t|_{U_\lambda} = \{\omega^{i,\lambda}_t=0=\theta^\lambda\}$ and $\calE|_{U_\lambda}=\ker\theta^\lambda$. Consider a partition of unity $\{\rho_\lambda\}$ subordinate to the covering. Set, $$X_t = \sum_\lambda \rho_\lambda X^\lambda_t$$
Then $X_t$ is a global field and $X_t\in\calL$ since each $X^\lambda_t\in \calL$. By Proposition~\ref{propDiffEqLocalIndept}, the local fields $X^\lambda_t$ satisfies
\[\iota_{X^\lambda_t}d\omega^{i,\mu}_t|_{\calD_t} + \frac{d}{dt}\omega^{i,\mu}_t|_{\calD_t} = 0, \forall i=1,\ldots,k\]
on $U_\lambda\cap U_\mu$, whenever the set is non-empty. Then the global field $X_t$ satisfies the hypothesis of Proposition~\ref{propDiffEquationLocal} over each $U_\lambda$, as $X_t$ is a convex linear combination of the local fields.\\

Now in order to prove Theorem~\ref{thmIsotopyOfD}, we need to find \emph{local} field $X_t\in \calL$ such that
$$\iota_{X_t}d\omega^i_t|_{\calD_t} + \frac{d}{dt}\omega^i_t|_{\calD_t} = 0, \forall i=1,\ldots,k$$
on some open set $U$, where $\omega^i_t,\theta$ as in Proposition~\ref{propDiffEquationLocal}

\subsubsection{Obtaining the time-dependent field}
In order to obtain the local field $X_t$, we introduce a few notations. Define,
\begin{align*}
\calK^i_t &= \ker d\omega^i_t|_{\calD_t}=\Big\{X\in\calD_t\Big| d\omega^i_t(X,Y)=0, \forall Y\in\calD_t\Big\}\\
&\phantom{= \ker d\omega^i_t|_{\calD_t}\hspace{0.8ex}}=\Big\{X\in\calD_t\Big| \omega^i_t([X,Y])=0, \forall Y\in\calD_t\Big\}\\
\calJ^i_t &= \bigcap_{j\ne i} \calK^j_t=\Big\{X\in\calD_t\Big| \omega^j_t([X,Y])=0,\forall Y\in\calD_t,\forall j\ne i\Big\}\\
\calW_t &= \bigcap_j \calK^j_t = \Big\{X\in\calD_t \Big| [X,Y] \in \calD_t, \quad \forall Y\in \calD_t\Big\}
\end{align*}

\begin{lemma}
For each $i=1,\ldots,k$ we have $\calK^i_t\subset \calL,\forall t$
\end{lemma}
\begin{proof}
Locally, we have a family of $1$-forms $\alpha_t$, such that $\calL = \calD_t \cap \ker\alpha_t$. Since $\calL$ is integrable, from Frobenius theorem we have, in particular, $\alpha_t\wedge\omega^1_t\wedge\ldots\wedge\omega^k_t\wedge d\omega^i_t = 0, i=1,\ldots,k$. Pick some $K\in\calK^i_t$. Then we have,
$$0=\iota_K(\alpha_t\wedge\omega^1_t\wedge\ldots\wedge\omega^k_t\wedge d\omega^i_t) =\alpha_t(K)\omega^1_t\wedge\ldots\wedge\omega^k_t\wedge d\omega^i_t$$
as the other terms vanish. But $\omega^1_t\wedge\ldots\wedge\omega^k_t\wedge d\omega^i_t\ne 0$. Hence we have, $\alpha_t(K) = 0$ and hence $K\in\calL$. Thus we have, $\calK^i_t\subset\calL,\forall i,\forall t$
\end{proof}
In particular we have that $\calJ^i_t\subset\calL$ for each $i=1,\ldots,k$ and $\calW_t\subset\calL$. Also observe that for any $i$, $\calW_t = \calJ^i_t \cap \calK^i_t$.
\begin{lemma}
\label{propRankJ}
For any $0\le t\le 1$, $\calK^i_t, \calW_t$ and $\calJ^i_t$ have constant ranks. Further, $\rk\calJ^i_t = \rk\calW_t + 1$ and co-rank of $\calK^i_t$ in $\calL$ is $1$
\end{lemma}
\begin{proof}
Fix $0\le t\le 1$. Since $\calL$ has co-rank $1$ in $\calD_t$, choose a vector field $Z$ such that $\calD_t = \calL \oplus \langle Z \rangle$. Then consider the map,
\begin{align*}
\Psi : \calL &\to \calE/\calD_t\\
L &\mapsto [Z,L]\mod\calD_t
\end{align*}
$\Psi$ is a bundle map. Also $\Psi$ has full rank, since $\calD_t^2=\calE$ and $\calL$ is integrable. Clearly $\calW_t$ is contained in the kernel of $\Psi$. Also for $\Psi(L)=0$, i.e., $[Z,L]\in\calD_t$ we have that $[L,X]\in\calD_t$ for any $X\in\calD_t$, since $\calL$ is integrable. Thus $L\in\calW_t$ and so $\calW_t=\ker\Psi$. Hence $\calW_t$ is a constant rank distribution.

Since the induced forms $\omega^i_t|_{\calE/\calD_t}$ are (point-wise) linearly independent, choose some vector fields $\{V_i\}$ from $\calE_t$, such that $\{\bar{V}_i=V_i \mod \calD_t\}$ is the corresponding dual basis. Consider the map,
\begin{align*}
\Psi_i : \calL &\to \calE/(\calD_t\oplus\langle V_i\rangle)\\
L &\mapsto [Z,L]\mod (\calD_t\oplus\langle V_i\rangle)
\end{align*}
Again $\Psi_i$ is a full rank bundle map. As $\bar{V}_i$ is dual to $\omega^i_t|_{\calE/\calD_t}$, for any $L\in\calJ^i_t$, we have that $$[Z,L] \mod\calD_t = f_i \bar{V_i}$$ for some function $f_i$, and thus $\calJ^i_t\subset\ker\Psi_i$. Conversely, suppose $\Psi_i(X)=0$,i.e., $[Z,X] \in \calD_t \oplus \langle V_i\rangle$. But then for $j\ne i$, $\omega^j_t[Z,X] = 0$ and so $X\in \calJ^i_t$. Thus $\calJ^i_t=\ker\Psi_i$, proving that $\calJ^i_t$ is of constant rank. Similarly define the map,
\begin{align*}
\Phi_i : \calL &\to \calE/(\calD_t\oplus\langle V_1,\ldots,\hat{V_i},\ldots,V_k\rangle)\\
L &\mapsto [Z,L]\mod (\calD_t\oplus\langle V_1,\ldots,\hat{V_i},\ldots,V_k\rangle)
\end{align*}
and observe that $\ker\Phi_i = \calK^i_t$. Since $\Phi_i$ is again a full rank bundle map, $\calK^i_t$ is of constant rank.

Clearly we have that $\rk\calJ^i_t = \rk\calW_t + 1$ and co-rank of $\calK^i_t$ in $\calL$ is $1$, for any $0\le t\le 1$.
\end{proof}

Now we find the local field.
\begin{lemma}
For each $i$ there is a local field $X^i_t\in\calJ^i_t$ such that $\iota_{X^i_t}d\omega^i_t|_{\calD_t} + \frac{d}{dt}\omega^i_t|_{\calD_t} = 0$
\end{lemma}
\begin{proof}
From Lemma~\ref{propRankJ} we have that $\calJ^i_t = \calW\oplus\calU^i_t$ for some line field $\calU^i_t\subset\calJ^i_t$. Clearly $\calU^i_t\not\subset\calK^i_t$. Then we can get, $\calD_t = \calK^i_t \oplus \calV^i_t$ such that $\calU^i_t\subset\calV^i_t, \forall t$. As co-rank of $\calK^i_t$ in $\calL$ is $1$ we have that $\calV^i_t$ is of constant rank with $\rk\calV^i_t = 2$. Since by definition $\calK^i_t =\ker d\omega^i_t|_{\calD_t}$, we have that $d\omega^i_t$ is non-degenerate over $\calV^i_t$. Hence we have a solution $X^i_t\in\calV^i_t$ such that, $$\iota_{X^i_t}d\omega^i_t|_{\calV^i_t} + \frac{d}{dt}\omega^i_t|_{\calV^i_t} = 0$$
Now, $\iota_{X^i_t}d\omega^i_t|_{\calK^i_t}=0$. Also since $\calK^i_t\subset\calL\subset\calD_s$ for every parameter $s$, we have that $\frac{d}{dt}\omega^i_t|_{\calK^i_t}=0$, since $\calD_s\subset\ker\omega^i_s$. Thus we also have that, $\iota_{X^i_t}d\omega^i_t|_{\calK^i_t} + \frac{d}{dt}\omega^i_t|_{\calK^i_t} = 0$. Combining we get that $$\iota_{X^i_t}d\omega^i_t|_{\calD_t} + \frac{d}{dt}\omega^i_t|_{\calD_t} = 0$$
as required. We now show that $X^i_t\in\calU^i_t$, which will yield that $X^i_i\in\calJ^i_t$

Restricting to $\calU^i_t$ we see that, $\iota_{X^i_t}d\omega^i_t|_{\calU^i_t} =- \frac{d}{dt}\omega^i_t|_{\calU^i_t}$. But $\calU^i_t\subset\calJ^i_t\subset\calL\subset\calD_s\subset\ker\omega^i_s,\forall s$ and so $\frac{d}{dt}\omega^i_t|_{\calU^i_t}=0$. Thus we have that, $$\iota_{X^i_t}d\omega^i_t|_{\calU^i_t}=0$$
Now $X^i_t\in\calV^i_t$ and $\calU^i_t\subset\calV^i_t$ is a line field. But $d\omega^i_t$ is non-degenerate on $\calV^i_t$, with $\rk\calV^i_t=2$. Thus the only way $\iota_{X^i_t}d\omega^i_t|_{\calU^i_t}=0$ is possible if $X^i_t\in\calU^i_t$. This completes the proof.
\end{proof}

Set $$X_t=\sum_i X^i_t.$$ Since each $X^i_t\in \calJ^i_t\subset\calL$, we have that $X_t\in\calL$.
\begin{lemma}
$\iota_{X_t}d\omega^i_t|_{\calD_t} + \frac{d}{dt}\omega^i_t|_{\calD_t} = 0,\forall i=1,\ldots,k$
\end{lemma}
\begin{proof}
Since $X^i_t\in\calJ^i_t$, we have that $\iota_{X^i_t}d\omega^j_t|_{\calD_t} = 0\forall j\ne i$. Thus,
\begin{align*}
\iota_{X_t}d\omega^i_t|_{\calD_t} &= \sum_j \iota_{X^j_t}d\omega^i_t|_{\calD_t}\\
&=\iota_{X^i_t}d\omega^i_t|_{\calD_t}\\
&=-\frac{d}{dt}\omega^i_t|_{\calD_t}
\end{align*}
that is to say, $\iota_{X_t}d\omega^i_t|_{\calD_t} + \frac{d}{dt}\omega^i_t|_{\calD_t} = 0$ for every $i=1,\ldots,k$.
\end{proof}

\subsection{Proof of Theorem~\ref{thmGrayGeneral}}

We shall see that Theorem~\ref{thmGrayGeneral} follows from Theorem~\ref{thmIsotopyOfD}
by using the following lemma.
\begin{lemma}
\label{lemmaIsotopyOfE}
Suppose we are given a one-parameter family of co-rank $1$ distributions $\calE_t$ on a compact manifold $M$, such that Cauchy characteristic distribution $\calL_t$ of $\calE_t$ is independent of $t$, say $\calL_t=\calL$ and $TM=\calE_t^2$. Then there exists an isotopy $\phi_t$ of $M$ such that $$\phi_{t*}\calE_0 = \calE_t,\qquad \phi_{t*} \calL = \calL$$
\end{lemma}

\begin{proof}[Proof of Theorem~\ref{thmGrayGeneral}]
Since $\calD_t^3=TM$, in particular we have that $TM = \calE_t^2$. Hence, using Lemma~\ref{lemmaIsotopyOfE} we get an isotopy $\phi_t$ that fixes $\calL$ and $\phi_{t*}\calE_0=\calE_t$. Since $\phi_t$ is a diffeomorphism, we get $\phi_{t*}^{-1}\calE_t=\calE_0$. Set, $\calD_t^\prime =\phi_{t*}^{-1}\calD_t$. Clearly we have the flag, $\calL\subset\calD_t^\prime\subset\calE_0$, where $\calL$ is the Cauchy characteristic distribution of $\calE_0$. Since Lie brackets are preserved under push-forwards by diffeomorphisms, we have that $\calD_t^{\prime2}=\calE_0$ and $\calD_t^{\prime3} = TM$.

Now using Theorem~\ref{thmIsotopyOfD}, we get another isotopy $\psi_t$ that fixes both $\calL$ and $\calE_0$, and $\psi_{t*}\calD_0=\calD_t^\prime$. Then, $\psi_{t*}^{-1}\calD_t^\prime=\calD_0$. Set, $\Phi_t = \psi_t^{-1} \circ \phi_t^{-1}$. Then $\Phi_t$ is the desired isotopy such that,
$$\Phi_{t*}\calD_t = \calD_0,\qquad \Phi_{t*}\calL = \calL$$
\end{proof}

\subsubsection{Proof of Lemma~\ref{lemmaIsotopyOfE}}
The proof can be found in \cite{montGoursat}. We reproduce it here with minor modification.\\

We have, $\calE_t \underset{loc.}{=} \ker \theta_t$. By hypothesis, $\calL = \ker d\theta_t|_{\calE_t},\forall t$. Suppose, $\calE_t = \calL \oplus \calV_t$, where $\calV_t=\calL^{\perp_{\calE_t}}$ with respect to some fixed choice of a Riemannian metric. Then $d\theta_t$ is non-degenerate on the sub-space $\calV_t$. So there exists a unique (local) field $X_t\in V_t$ such that $$\iota_{X_t} d\theta_t|_{\calV_t} = - \frac{d}{dt}\theta_t|_{\calV_t}$$
Since $X_t\in\calE_t$ and $\calL = \ker d\theta_t|_{\calE_t}$, we have $\iota_{X_t}d\theta_t|_\calL = 0$. Also, since $\calL \subset\calE_s=\ker\theta_s, \forall s$, we have that $\frac{d}{dt}\theta_t|_\calL = 0$. Then combining the two we have, $$\iota_{X_t} d\theta_t|_{\calE_t} = - \frac{d}{dt}\theta_t|_{\calE_t}$$
As $\iota_{X_t} d\theta_t + \frac{d}{dt}\theta_t=0$ when restricted to $\calE_t$, we have that $$\iota_{X_t} d\theta_t + \frac{d}{dt}\theta_t=h_t\theta_t$$
for some family of functions $h_t$.\\

Now integrating $X_t$ we get a local flow $\phi_t$. Then,
\begin{align*}
\frac{d}{dt} \phi_t^* \theta_t &= \phi_t^*\Big(L_{X_t}\theta_t + \frac{d}{dt}\theta_t\Big)\\
&= \phi_t^*\Big(\iota_{X_t}d\theta_t + \frac{d}{dt}\theta_t\Big)\\
&= \phi_t^*(h_t\theta_t)\\
&= g_t\phi_t^*\theta_t
\end{align*}
where $g_t = h_t\circ\phi_t$. Now we would like to have, $\phi_t^*\theta_t = H_t\theta_0$ for some non-zero function $H_t$. Differentiating both sides, we get $$\frac{dH_t}{dt}\theta_0 = \frac{d}{dt}\phi_t^*\theta_t = g_t\phi_t*\theta_t = g_tH_t\theta_0$$
Thus we have the differential equation, $$\frac{dH_t}{dt} = g_tH_t$$ with the initial condition, $H_0 = 1$. This has a solution, $$H_t = \exp \int_{0}^{t} g_s ds$$
which is clearly non-zero. Thus for the field $X_t$, the local flow $\phi_t$ satisfies $$\phi_{t*}\calE_0 = \calE_t$$ and fixes the Cauchy characteristic distribution $\calL$.\\

Now in order to get a global flow, we use partition of unity. First note that the vector field $X_t$ doesn't depend on the choice of the local defining form $\theta_t$. Consider $\eta_t = f_t\theta_t$ for some non-zero function $f_t$. Then, $\calE_t = \ker\eta_t$. Now, $d\eta_t = df_t\wedge\theta_t + f_t\wedge d\theta_t$. So,
\begin{align*}
\iota_{X_t} d\eta_t|_{\calE_t} &= \iota_{X_t}df_t \theta_t|_{\calE_t} - \iota_{X_t}\theta_tdf_t|_{\calE_t} + f_t\iota_{X_t}d\theta_t|_{\calE_t}\\
&= f_t\iota_{X_t}d\theta_t|_{\calE_t}
\end{align*}
since $\theta_t(X_t)=0$ and $\theta_t|_{\calE_t}=0$. Also, $\frac{d}{dt}\eta_t|_{\calE_t} = \frac{df_t}{dt}\theta_t + f_t\frac{d}{dt}\theta_t|_{\calE_t}=f_t\frac{d}{dt}\theta_t|_{\calE_t}$. Hence we have, $$\iota_{X_t}d\eta_t|_{\calE_t} + \frac{d}{dt}\eta_t|_{\calE_t} = 0$$
Thus $X_t$ only depends on $\calE_t$.\\

Now get a locally finite open cover $\{U_\lambda\}$ of $M$ such that $\calE_t|_{U_\lambda} = \ker\theta_t^\lambda$ for local form $\theta_t^\lambda$. Get unique local fields $X_t^\lambda$ as above. Also consider a partition of unity $\{\rho_\lambda\}$ subordinate to the open cover. Set, $$X_t = \sum_{\lambda} \rho_\lambda X_t^\lambda$$
Then $X_t$ is global vector field. Since $M$ is closed, integrating $X_t$ we get a global isotopy $\Phi_t$ on $M$. Also $X_t$ satisfies $\iota_{X_t}d\theta_t^\lambda + \frac{d}{dt}\theta_t^\lambda = 0$, for any $\lambda$. Thus $\Phi_t$ is the required isotopy of $M$ such that $$\Phi_{t*}\calE_0 = \calE_t.$$

\section{Normal Forms}
In this section we shall obtain the normal form of the generators of the Pfaffian system defining a generalized Engel structure $\mathcal D$. Suppose, $\calD$ is of co-rank $k+1$, where $k=2l+1$, and $\calL\subset\calD\subset\calE\subset TM$ is the associated with the canonical flag on $M$. Suppose $\calS_0$ and $\calS_1$ are the Pfaffian systems annihilating $\calD$ and $\calE$ respectively. Since $\calD\subset\calE$, we have that $\calS_0\supset\calS_1$. Suppose $\calS_1=\langle \theta\rangle$ and $\calS_0=\langle \theta,\omega^1,\ldots,\omega^k\rangle$ locally. Then by Proposition \ref{propDistToForms} we get certain relations among these forms. We want to get a standard normal forms for some basis of $\calS_0$ and $\calS_1$.\\

Since $\theta\wedge d\theta^{l+1}\ne 0$ and $\theta\wedge d\theta^{l+2}= 0$, around a point $p\in M$ we have some co-ordinate system (Theorem 3.1 in \cite{bryantExteriorDiff}) such that $\calE$ is the kernel of $$\Theta= dz - \sum_{i=1}^{l+1} x_{i+l+1}dx_i$$
Clearly $\calS_1=\langle\Theta\rangle$ and $\{\Theta, \omega^i\}$ is a Pfaffian system associated to the given generalized Engel structur. From $\omega^i\wedge\Theta\wedge d\Theta^{l+1}=0,\forall i$, we have that $0=\omega^i\wedge dx_1\wedge\ldots\wedge dx_{2l+2}\wedge dz$. Hence, there exist functions $a^{ij},b^i$ such that, $$\omega^i = \sum_{j=1}^{2l+2} a^{ij} dx_j + b^idz$$
Now we have, in particular, $\omega^1\wedge\ldots\wedge\omega^k\wedge\Theta \ne 0$. Therefore, the matrix
\[\begin{pmatrix}
a^{11} &\ldots &a^{2l+1,1} &-x_{l+2}\\
\vdots &\vdots &\vdots &\vdots\\
a^{1,l+1} &\ldots &a^{2l+1,l+1} &-x_{2l+2}\\
a^{1,l+2} &\ldots &a^{2l+1,l+2} &0\\
\vdots &\vdots &\vdots &\vdots\\
a^{1,2l+2} &\ldots &a^{2l+1,2l+2} &0\\
b^1 &\ldots &b^{2l+1} &1
\end{pmatrix}_{(2l+3) \times (2l+2)}\]
has full rank $2l+2$ everywhere. Evaluating this at the point $p$, we observe that the last column is $\begin{pmatrix}0 &\ldots &0 &1\end{pmatrix}^t$ and hence the last row cannot be linearly dependent on the rest. Hence, without loss of generality we may assume that the $(2l+2)^{\textnormal{th}}$ row is linearly dependent and the rest of them are linearly independent about $p$. Thus we can transform the matrix into,
\[\begin{pmatrix}
I_{(2l+1)\times (2l+1)} &0\\
C &0\\
0 &1
\end{pmatrix}\]
where $C=\begin{pmatrix}c_1 &\ldots &c_{2l+1}\end{pmatrix}$ is a row vector of functions. Thus we can transform $\omega^i$ into $$\Omega^i = dx_i + c_idx_{2l+2},\forall i=1,\ldots,k=2l+1$$
Clearly, $\{\Theta,\Omega^i\}$ still forms a basis of $\calS_0$. Hence we have $\mu^i=\Omega^1\wedge\ldots\wedge\Omega^k\wedge\Theta\wedge d\Omega^i\ne0$ which gives, $$dx_1\wedge\ldots\wedge dx_{2l+2}\wedge dz\wedge dc_i \ne 0,\forall i$$
But $\{\mu^i\}$ is also point-wise linearly independent. Hence we have that $\{dc_i\}$ are point-wise linearly independent as well i.e., $\{c_i\}$ are coordinate functions around $p$. Set $y_i=x_{i+l+1}$. Thus, we obtain a coordinate system
\[\big(x_1,\ldots,x_{l+1},y_1,\ldots,y_{l+1},z,c_1,\ldots,c_k,q_1,\ldots,q_r\big)\]
around $p$ such that the $1$-forms $\{\Theta,\Omega^i\}$ can be expressed as follows :

\begin{align*}
\Theta &= dz -\sum_{i=1}^{l+1} y_idx_i\\
\Omega^i &=
\begin{cases}
dx_i + c_idy_{l+1}, 1\le i \le l+1\\
dy_{i-l-1} + c_idy_{l+1}, l+2\le i\le k= 2l+1
\end{cases}
\end{align*}
where $r$ is such that $\dim M = r + 2k + 2$, such that $\calE=\ker\Theta$ and $\calD=\{\Omega^1=\ldots\Omega^k=0=\Theta\}$ Note that, by taking $l=0$ in the above normal form we obtain the normal form of the Pfaffian system defining an Engel structure, namely $dz-ydx$, $dx-wdy$.

\section*{Acknowledgment}
The author would like to thank Mahuya Datta, Dheeraj Kulkarni and Suvrajit Bhattacharjee for fruitful and enlightening discussions.

\nocite{hirschBook}
\bibliographystyle{alpha}
\bibliography{../Engel}

\end{document}